\documentclass[10pt]{amsart}
\usepackage{amscd,amssymb,epsfig}
\usepackage{graphics}
\newtheorem{theorem}{Theorem}[section]
\newtheorem{lemma}[theorem]{Lemma}
\newtheorem{proposition}[theorem]{Proposition}
\newtheorem{corollary}[theorem]{Corollary}
\theoremstyle{definition}

\theoremstyle{remark}

\numberwithin{figure}{section}
\numberwithin{table}{section}

\makeatletter

\@addtoreset{equation}{section}
\makeatother

\newcommand\bC{{\mathbb C}}
\newcommand\bR{{\mathbb R}}

\newcommand\Cg{{{\mathbb C}_g}}

\newcommand\clll{{c^\lambda}}
\newcommand\cllm{\overline{c^\mu}}

\newcommand\dc{\overline{\partial}}
\newcommand\dd{{\partial}}

\newcommand\elll{{\ell^\lambda}}
\newcommand\ellm{\overline{\ell^\mu}}
\newcommand\fiber{{\int_{\text{fiber}}}}
\newcommand\harmonic{{\mathcal{H}}}

\newcommand\hatM{{\widehat{M}}}
\newcommand\hatPhi{{\widehat{\Phi}}}

\newcommand\hdot{{\stackrel{\centerdot}{h}}}
\newcommand\LaH{{{\bigwedge}^3H}}

\newcommand\llll{{L^\lambda}}
\newcommand\lllm{\overline{L^\mu}}

\newcommand\Mg{{{\mathbb M}_g}}
\newcommand\Mgone{{{\mathbb M}_{g, 1}}}
\newcommand\moduli{{{\mathbb M}_g}}
\newcommand\mubar{{\overline{\mu}}}
\newcommand\nudot{{\stackrel{\centerdot}{\nu}}}
\newcommand\numaru{{\stackrel{\circ}{\nu}}}

\newcommand\omegaone{{\omega_{(1)}}}
\newcommand\omegaonedot{{\stackrel{\centerdot}{\omega_{(1)}}}}

\newcommand\pione{{\pi^{\mbox{\tiny orb}}_1}}

\newcommand\pH{{\frak{p}^H}}
\newcommand\pU{{\frak{p}^U}}
\newcommand\qH{{\frak{q}^H}}
\newcommand\qU{{\frak{q}^U}}

\newcommand\stardot{{\stackrel{\centerdot}{\ast}}}
\newcommand\T{{\widehat{T}}}
\newcommand\tast{{\widetilde{\ast}}}

\newcommand\TCM{{T_{\mathbb{C}_g/\mathbb{M}_g}}}

\newcommand\wotimes{{\widehat{\otimes}}}

\begin{document}

\title[Johnson's homomorphisms and the Arakelov-Green function]
{Johnson's homomorphisms and \\
the Arakelov-Green function}

\author[Nariya Kawazumi]{Nariya Kawazumi}
\thanks{The author is partially supported
by  Grant-in-Aid for Scientific Research (A) (No.18204002), 
the Japan Society for Promotion of Sciences.
}
\address{Department of Mathematical Sciences\\
University of Tokyo \\Komaba, Tokyo 153-8914\\
Japan}
\email{kawazumi{\char'100}ms.u-tokyo.ac.jp}

\begin{abstract}
Let $\pi: {\mathbb C}_g \to {\mathbb M}_g$ 
be the universal family of 
compact Riemann surfaces of genus $g \geq 1$.
We introduce a real-valued function 
on the moduli space ${\mathbb M}_g$ and compute the first and 
the second variations of the function.
As a consequence we relate the Chern form of the 
relative tangent bundle $T_{{\mathbb C}_g/{\mathbb M}_g}$ 
induced by the Arakelov-Green function 
with differential forms on ${\mathbb C}_g$ induced by 
a flat connection whose holonomy gives 
Johnson's homomorphisms on the mapping class group.
\end{abstract}

\maketitle


\begin{center}
 Introduction
\end{center}

Let $\pi: {\mathbb C}_g \to {\mathbb M}_g$ be the universal family of 
compact Riemann surfaces of genus $g \geq 1$.
The orbifold fundamental group $\pione(\Mg)$ of the moduli space 
$\Mg$ is the mapping class group for a closed surface of genus $g$. 
Johnson \cite{J1} introduced a series of nested subgroups of 
$\pione(\Mg)$ and a homomorphism 
on each of these subgroups. Today they are called Johnson's 
homomorphisms. The first one of the subgroups is the Torelli 
group $\mathcal{I}_g$, the kernel of the action of 
the mapping class group on the homology of the surface.
Moreover he proved the free part of the abelianization 
${\mathcal{I}_g}^{\mbox{\tiny abel}}$ is given by 
the first Johnson homomorphism \cite{J2}. 
Morita extended the first Johnson homomorphism to 
a crossed homomorphism on the whole group $\pione(\Mg)$,
and proved that his extension yields all of the Morita-Mumford 
classes on the moduli space $\Mg$ \cite{MoJ}\cite{MoF}.\par
It is an important subject to study differential geometry 
of the moduli space $\Mg$ through Johnson's homomorphisms.
Harris \cite{Harris} defined the harmonic volume 
of a compact Riemann surface. This can be interpreted 
as an analytic couterpart of the first Johnson homomorphism. Let
$\mathcal{L}$ be the  Hodge line bundle on the moduli space $\Mg$. 
The first Morita-Mumford class $e_1$ is twelve times 
the Chern class $c_1(\mathcal{L})$. 
Hain and Reed took the pullback of 
the biextension line bundle \cite{Hain} along the harmonic volume 
to construct  a Hermitian line bundle $\mathcal{B}$ on $\Mg$, isomorphic 
to $\mathcal{L}^{\otimes (8g+4)}$. Comparing the Hermitian metric on
$\mathcal{B}$ with the standard metric on $\mathcal{L}$, they 
defined and studied 
a real-valued function $\beta_g: \Mg \to \bR$ \cite{HR}.\par
In our previous paper \cite{Kaw2} we introduced a flat 
connection on a vector bundle over the space 
$\Mgone := T_{{\mathbb C}_g/{\mathbb M}_g} 
\setminus \mbox{($0$-section)}$
whose holonomy gives all of Johnson's homomorphisms 
for the mapping class group $\pi_1(\Mgone)$.  
The first term of the connection form is exactly 
the first variation of (pointed) harmonic volumes. 
By Morita's recipe
\cite{MoF}  the connection form induces canonical $2$-forms 
$e^J$ on ${\mathbb C}_g$ and $e^J_1$ on ${\mathbb M}_g$
representing the Chern class of the relative tangent bundle
$T_{{\mathbb C}_g/{\mathbb M}_g}$ and the first Morita-Mumford class
$e_1$, respectively. 
The $2$-form $e^J_1$ corresponds to the Hermitian bundle $\mathcal{B}$. 
\par
Let $\{\psi_i\}^g_{i=1}$ be an orthonormal 
basis of the holomorphic $1$-forms; 
$\frac{\sqrt{-1}}{2}\int_C\psi_i\wedge\overline{\psi_j} 
= \delta_{ij}$, $1 \leq i, j \leq g$. 
Then the $2$-form 
$B :=  \frac{\sqrt{-1}}{2g}\sum^g_{i=1}\psi_i\wedge
\overline{\psi_i}$ is a volume form on $C$, 
independent of the choice of the orthonormal basis 
$\{\psi_i\}^g_{i=1}$. 
The $2$-form $e^J$ is related to the volume form $B$. 
It restricts to $B$ 
$$
e^J\vert_C = (2-2g)B \in A^2(C)
$$
on any Riemann surface 
$C$ regarded as a fiber of the universal family 
 $\pi: {\Bbb C}_g \to {\Bbb M}_g$ (\ref{403.65}). 
Here $A^q(\,\,)$ means the space of the $q$-forms. 
\par
Arakelov made use of the volume form $B$ in his study of 
arithmetic surfaces \cite{A}. 
The Arakelov-Green function 
$G$ on a compact Riemann surface $C$ is defined by
$G(P_0, P_1) := \exp(-4\pi h_{P_0}(P_1))$, $P_0, P_1 \in C$, 
where the function $h_{P_0}$ is the Green function with respect to the
volume form $B$ (\ref{107.1}).  
We may regard the Arakelov-Green function
$G$ as a function defined on the fiber product 
${\mathbb C}_g\times_{{\mathbb M}_g}{\mathbb C}_g$.  
Then the differential form $e^A := 
\frac{1}{2\pi\sqrt{-1}}\partial\overline{\partial}\log G
\vert_{\mbox{(diagonal)}}$ represents 
the Chern class of the line bundle $T_{{\mathbb C}_g/{\mathbb M}_g}$, 
since the normal bundle of the diagonal in the product 
${\mathbb C}_g\times_{{\mathbb M}_g}{\mathbb C}_g$ is 
exactly the relative tangent bundle 
$T_{{\mathbb C}_g/{\mathbb M}_g}$. 
As was observed by Arakelov \cite{A}, we have 
$$
e^A\vert_C = (2-2g)B \in A^2(C)
$$
for any $C$. \par
In this paper we will give an explicit function 
which links $e^J$ and $e^A$ together. We define 
$$
a_g(C) := -\sum^g_{i, j=1}\int_C\psi_i\wedge\overline{\psi_j}\,
\hatPhi(\overline{\psi_i}\wedge\psi_j)
$$
for any compact Riemann surface $C$,
where $\hatPhi: A^2(C) \to A^0(C)$ is the Green operator 
with respect to the volume form $B$ (\ref{106.1}). 
It is independent of the choice of
the orthonormal basis $\{\psi_i\}^g_{i=1}$. If $g \geq 2$, 
$a_g(C)$ is a positive real number (Corollary \ref{108.5}).
Hence we obtain a real-valued function $a_g: \Mg \to \bR$. 
Then we have 
\begin{theorem}[= Theorem \ref{300.1}]
$$
e^A - e^J = \frac{-2\sqrt{-1}}{2g(2g+1)}\dd\dc a_g.
$$
\end{theorem}
The $2$-form $e^J$ is induced by the exterior derivative 
of the second term of the flat connection. 
This fact simplifies the proof, which requires only the 
first variation of the function $a_g$. 
\par

Next we study the $2$-form $e_1^J$. 
Moreover we consider the integral along the fiber 
$e^F_1 := \int_{\text fiber}(e^J)^2 \in A^2(\Mg)$, which also 
represents the first Morita-Mumford class $e_1$. 
The difference $e^F_1 - e^J_1 \in A^{2}({\Bbb M}_g)$ is 
null-cohomologous, but does not vanish as a differential 
form. Then we have 
\begin{theorem}[= Theorem \ref{603.2}]
$$
\frac{-2\sqrt{-1}}{2g(2g+1)}\dd\dc a_g 
= \frac{1}{(2g-2)^2}(e^F_1 - e^J_1).
$$
\end{theorem}
To prove the theorem, we compute explicitly 
$e^F_1$, $e^J_1$ and the second variation of $a_g$. 
The latter half of this paper is devoted to these computations. 
\par
As a corollary of these two theorems we obtain 
\begin{corollary}
$$
e^A - e^J = \frac{1}{(2g-2)^2}(e^F_1 - e^J_1) \in A^2(\Cg).
$$
\end{corollary}
Thus the Chern form $e^A$ induced by the Arakelov-Green function 
is expressed in terms of differential forms induced
by the flat connection whose holonomy gives 
Johnson's homomorphisms on the mapping class group.\par

The author does not know any of further properties 
of the function $a_g$.
It would be interesting if we could find its explicit 
relations with other real-valued functions on $\Mg$
including Faltings' delta function \cite{F} and Hain-Reed's 
beta function \cite{HR}. 

\tableofcontents

\section{A real-valued function $a_g$ on the moduli space
$\moduli$.}\label{s-a}

We begin by recalling some basic notions on Riemann surfaces. 
In the second half of this section we define 
the real-valued function $a_g$ on the space $\moduli$. \par

Let $g \geq 1$ be an integer, 
$C$ a compact Riemann surface of genus $g$. \par

The Hodge $*$-operator $*: T^*_\bR C \to T^*_\bR C$ 
on the real cotangent bundle $T^*_\bR C$ depends only on the complex
structure of $C$. The $-\sqrt{-1}$-eigenspace in $\left(T^*_\bR
C\right)\otimes\bC$ is the holomorphic cotangent bundle
$T^*C$, and the $\sqrt{-1}$-eigenspace the antiholomorphic
cotangent bundle $\overline{T^*C}$. \par

We denote by $A^q(C)$ the complex-valued
$q$-currents on $C$ for $0 \leq q \leq 2$. 
The operator $*$ decomposes the space $A^1(C)$ into 
the $\pm \sqrt{-1}$-eigenspaces
$$
A^1(C) = A^{1, 0}(C) \oplus A^{0, 1}(C),
$$ 
where $A^{1, 0}(C)$ is the $-\sqrt{-1}$-eigenspace and 
$A^{0, 1}(C)$ the $\sqrt{-1}$-eigenspace. 
Throughout this paper we denote by  $\varphi'$ and $\varphi''$ 
the $(1, 0)$- and the $(0, 1)$-parts of $\varphi \in A^1(C)$, 
respectively, i.e., 
$$
\varphi = \varphi' + \varphi'', \quad 
*\varphi = -\sqrt{-1}\varphi' + \sqrt{-1}\varphi''.
$$
If $\varphi$ is harmonic, then $\varphi'$ is holomorphic and 
$\varphi''$ anti-holomorphic.\par

Let $H$ denote the complex first homology group of $C$,
$H_1(C; \bC)$, which admits the intersection pairing 
\begin{equation*} 
m: H\otimes H \to \bC, \quad
X\otimes Y \mapsto m(X\otimes Y) = X\cdot Y.
\end{equation*}
The dual $H^*$ is the first complex cohomology group 
$H^1(C; \bC)$. 
Consider the map $H^* \to A^1(C)$ assigning to each cohomology class the
harmonic $1$-form  representing it. The map can be regarded as a
$H$-valued  real harmonic $1$-form $\omegaone \in A^1(C)\otimes H$. \par
Let $\{X_i, X_{g+i}\}^g_{i=1}$ be a symplectic basis of 
$H$
$$
X_{i}\cdot X_{g+j} = \delta_{ij}, \quad
X_{i}\cdot X_{j} = X_{g+i}\cdot X_{g+j} = 0, \quad
1 \leq i, j \leq g,
$$
and  $\{\xi_i, \xi_{g+i}\}^g_{i=1} \subset A^1(C)$ 
the basis of the harmonic $1$-forms dual to  $\{X_i, X_{g+i}\}^g_{i=1}$. Then we have 
\begin{equation*}
\omegaone = \sum^g_{i=1}\xi_{i}X_{i} + \xi_{g+i}X_{g+i} 
\in A^1(C)\otimes H.
\end{equation*}
In particular, if $\{\psi_i\}^g_{i=1}$ is an orthonormal basis of the holomorphic $1$-forms
\begin{equation}
\frac{\sqrt{-1}}{2}\int_C\psi_i\wedge\overline{\psi_j} 
= \delta_{ij}, \quad 1 \leq i, j \leq g, 
\label{104.1}
\end{equation}
then we obtain
\begin{equation}
\omegaone = \sum^g_{i=1} \psi_iY_i + \overline{\psi_i}\overline{Y_i},
\label{104.2}
\end{equation}
where $\{Y_i, Y_{g+i}\}^g_{i=1} \subset H_\bC$ is the dual basis 
of the symplectic basis $\{[\psi_i],
\frac{\sqrt{-1}}{2}[\overline{\psi_i}]\}^g_{i=1}$ 
of $H^* = H^1(C; \bC)$.\par

Since the complete linear system of the canonical divisor 
on the complex algebraic curve $C$ has no basepoint, 
the $2$-form 
\begin{equation}
B := \frac{1}{2g}m(\omegaone\wedge\omegaone) = \frac{1}{2g}\omegaone\cdot\omegaone = 
\frac{\sqrt{-1}}{2g}\sum^g_{i=1}\psi_i\wedge\overline{\psi_i} 
\label{105.1}
\end{equation}
is a volume form on $C$ with $\int_CB = 1$.\par

We denote by $\hatPhi = \hatPhi_C: A^2(C) \to A^0(C)$ the Green operator with respect to the volume form $B$.
We have
\begin{eqnarray}
d\ast d\hatPhi(\Omega) &=& \Omega - (\int_C\Omega)B,\label{106.1}\\
\int_C\hatPhi(\Omega)B &=& 0\label{106.2}
\end{eqnarray}
for any $\Omega \in A^2(C)$. 
The Hodge decomposition on the $1$-currents on $C$ 
is given by 
\begin{equation}
\varphi = \harmonic\varphi + \ast d\hatPhi d\varphi + d\hatPhi d\ast \varphi
\label{106.3}\end{equation}
for any $\varphi \in A^1(C)$.
Here $\harmonic$ is the harmonic projection and 
satisfies 
\begin{equation}
\harmonic \varphi =
-\omegaone\cdot\left(\int_C\omegaone\wedge\varphi\right) 
= -\left(\int_C\varphi\wedge\omegaone\right)
\cdot\omegaone.
\label{106.4}\end{equation}
If $\varphi'$ is a $(1, 0)$-current, then 
\begin{equation}
\varphi' = \harmonic\varphi' + 2\ast\dd\Psi\dc\varphi'
= \harmonic\varphi' - 2\sqrt{-1}\dd\Psi\dc\varphi'.
\label{106.5}
\end{equation}

Let $\delta_{P_0} \in A^2(C)$ be the delta current at $P_0 \in C$. 
We define 
\begin{equation}
h = h_{P_0} := - \hatPhi(\delta_{P_0}).
\label{107.1}\end{equation}
$G(P_0, P_1) := \exp(-4\pi h_{P_0}(P_1))$ is the 
Arakelov-Green function. We have
\begin{equation*}
\frac{1}{2\pi\sqrt{-1}}\dd\dc\log G(P_0, \empty) 
= B - \delta_{P_0}.
\end{equation*}
If $\Omega$ is a smooth $2$-form on $C$, then 
$\hatPhi(\Omega)$ is smooth. We have 
\begin{equation*}
\int_C\hatPhi(\Omega')\Omega = \int_C\Omega'\hatPhi(\Omega)
\end{equation*}
for any $\Omega' \in A^2(C)$. In particular, 
if $\Omega' = \delta_{P_0}$, then 
\begin{equation*}
\hatPhi(\Omega)(P_0) = - \int_Ch_{P_0}\Omega = 
\frac1{4\pi}\int_C\log G(P_0, \empty)\Omega.
\end{equation*}

Now we introduce the function $a_g$ on the moduli space 
$\moduli$. For $[C] \in \moduli$ we define
\begin{equation*}
a_g(C) := (m\otimes m)\int_C\omegaone\hatPhi(\omegaone\wedge\omegaone)\omegaone
= \int_C\omegaone\cdot\hatPhi(\omegaone\wedge\omegaone)\cdot\omegaone.
\end{equation*}
From (\ref{104.2})
\begin{eqnarray*}
a_g(C) &=& -\sum^g_{i, j=1}\int_C\psi_i\wedge\overline{\psi_j}\,
\hatPhi(\overline{\psi_i}\wedge\psi_j)
\\
&=& -\frac{1}{4\pi}\sum^g_{i, j=1}\int_{(P_0, P_1) \in C\times C}(\psi_i\wedge\overline{\psi_j})_{P_0}\,
\log G(P_0, P_1)(\overline{\psi_i}\wedge\psi_j)_{P_1}.
\end{eqnarray*}
It is easy to see 
\begin{lemma}\label{108.4} We have 
$$
\int_C\Omega\hatPhi(\overline{\Omega}) \leq 0
$$
for any smooth $2$-form $\Omega$. 
Moreover we have $\int_C\Omega\hatPhi(\overline{\Omega}) = 0$ 
if and only if $\Omega \in \bC B$.
\end{lemma}
\begin{proof}
By a straightforward computation we have
$$
\int_C\Omega\hatPhi(\overline{\Omega}) 
= -\sqrt{-1}\int_C \dd\hatPhi(\Omega)
\wedge\dc\hatPhi(\overline{\Omega})
- \sqrt{-1}\int_C \dd\hatPhi(\overline{\Omega})
\wedge\dc\hatPhi(\Omega).
$$
This implies the first half of the lemma.\par
Assume $\int_C\Omega\,\hatPhi(\overline{\Omega}) = 0$. 
Then $\dd\hatPhi(\Omega) = \dc\hatPhi(\Omega)
= 0$, so that $\hatPhi(\Omega)$ is constant.
From (\ref{106.1}) we obtain $\Omega = (\int_C\Omega)B \in \bC B$.
\end{proof}

It follows from the lemma
\begin{corollary}\label{108.5}
$$
a_g(C) > 0, \quad\mbox{if $g \geq 2$.}
$$
\end{corollary}
In the case $g=1$ we have $a_1(C) = 0$.

\section{The first variation of the function $a_g$.}\label{s-fv}
In this section we study the first variations of the funcion $a_g$
and the Green function $h$.\par

Let $C$ be a compact Riemann surface of genus $g \geq 1$. 
We define the map $M: H^{\otimes 4}\to \bC$ by 
\begin{equation*}
M(Z_1Z_2Z_3Z_4) := (m\otimes m)(Z_2Z_3Z_4Z_1) = 
 (Z_2\cdot Z_3)(Z_4\cdot Z_1)
\end{equation*}
for $Z_i \in H$. Here and throughout this paper we omit 
the symbol $\otimes$ frequently, so that we write simply 
$Z_1Z_2Z_3Z_4$ for $Z_1\otimes Z_2\otimes Z_3\otimes Z_4$. 
Then 
\begin{equation*}
a_g(C) = -M\int_C\hatPhi(\omegaone\wedge\omegaone)
\omegaone\wedge\omegaone
\end{equation*}
Let $I \in H^{\otimes 2}$ denote the intersection form. 
If $\{X_i, X_{g+i}\}^g_{i=1}\subset H$ is a symplectic basis, then 
\begin{equation}
I = \sum^g_{i=1}X_{i} X_{g+i} - X_{g+i} X_{i}.
\label{201.3}
\end{equation}
For $Z, Z_1, Z_2 \in H$ it is easy to see
\begin{eqnarray}
&&(m\otimes 1_H)(ZI) = (1_H\otimes m)(IZ) = -Z,\label{201.4}\\
&& M(IZ_1Z_2) = M(Z_1Z_2I) = m(Z_1Z_2) = Z_1\cdot Z_2.\label{201.5}
\end{eqnarray}
We denote by $\tast: H \to H$ the transpose of the Hodge $*$-operator on $H^* = H^1(C; \bC)$. 
It is clear that 
\begin{equation}
(\tast Z_1)\cdot (\tast Z_2) = Z_1\cdot Z_2, \quad 
\forall Z_1, \forall Z_2 \in H.
\label{201.6}
\end{equation}
We have 
$\tast Y_i = -\sqrt{-1}Y_i$ for the symplectic basis 
$\{Y_i, Y_{g+i}\}^g_{i=1}$ in (\ref{104.2}). 
Let $H'$ and $H''$ be the $-\sqrt{-1}$- and 
the $\sqrt{-1}$- eigenspaces of the operator 
$\tast$, respectively. 
These are isotropic subspaces of $H$. 
We have
\begin{equation}
\omegaone = \omega'_{(1)} + \omega''_{(1)}, \quad
\omega'_{(1)} = \overline{\omega''_{(1)}} = \sum^g_{i=1}\psi_iY_i 
\in A^1(C)\otimes H',
\label{201.7}
\end{equation}
where $\{\psi_i\}^g_{i=1}$ is the orthonormal basis 
of the holomorphic $1$-forms in (\ref{104.1}). In particular,
\begin{equation}
\ast\omegaone = \tast\omegaone.
\label{201.8}
\end{equation}

Consider a $C^\infty$ family of compact Riemann surfaces 
$C_t$, $t \in \bR$, $\vert t\vert\ll 1$, with 
$C_0 = C$. The family $\{C_t\}$ is trivial as a 
$C^\infty$ fiber bundle over an interval near $t=0$, 
so that we obtain a $C^\infty$ family of $C^\infty$ diffeomorphisms 
$f^t: C\to C_t$ with $f^0 = 1_C$.  
In general, if $\bigcirc = \bigcirc_t$
is a ``function" in 
$t \in \bR$, $\vert t \vert \ll 1$, then we write simply 
\begin{equation}
\mathop\bigcirc^{\centerdot} := 
\frac{d}{dt}\Bigr\vert_{t=0}\bigcirc_t.
\label{202.1}\end{equation}
For example, we denote
$$
(\mu(f))^\centerdot := \frac{d}{dt}\Bigr\vert_{t=0}\mu(f^t).
$$
Here $\mu(f^t)$ is the complex dilatation of the diffeomorphism $f^t$. 
Let $z$ be a complex coordinate on $C$, and $\zeta$ on $C_t$. 
The complex dilatation $\mu(f^t)$ is defined locally by 
\begin{equation*}
\mu(f^t) = \mu(f^t)(z)\frac{d}{dz}\otimes d\overline{z} = 
\frac{(\zeta\circ
f^t)_{\overline{z}}}{(\zeta\circ
f^t)_{z}}\frac{d}{dz}\otimes d\overline{z},
\end{equation*}
which does not depend on the choice of the coordinates $z$ and $\zeta$. 
In this paper we write more simply
$$
\mu = (\mu(f))^\centerdot.
$$
The Dolbeault cohomology class $[\mu] \in H^1(C; \mathcal{O}_C(TC))$ 
is exactly the tangent vector $\frac{d}{dt}\bigr\vert_{t=0}[C_t] \in
T_{[C]}\moduli$. \par

We define a linear operator $S = S[\mu]: A^1(C) \to A^1(C)$ by 
$$
S(\varphi) = S(\varphi') + S(\varphi'') :=
-2\varphi'\mu - 2\varphi''\overline{\mu},
$$
for $\varphi = \varphi' + \varphi''$,  $\varphi' \in A^{1, 0}(C)$, 
$\varphi'' \in A^{0, 1}(C)$. 
The first variation of $\ast_t := {f^t}^*\ast_{C_t}$ is given by 
\begin{equation}
\stardot = *S = -S*: A^1(C) \to A^1(C).
\label{202.2}
\end{equation}
\cite[\S7]{Kaw2}.\par
As to the harmonic form $\omega^t_{(1)} := {f^t}^*\omega^{C_t}_{(1)}$, we have 
\begin{lemma}\label{202.3}
$$
\omegaonedot = -d\hatPhi d* S\omegaone.
$$
\end{lemma}
\begin{proof}
Since $\omega^t_{(1)}$ is cohomologous to $\omega^0_{(1)}$, 
there exists a function $u$ such that $\omegaonedot =
du$. Differentiating 
$d\ast_t\omega^t_{(1)} = 0$, we get
$$
d\ast du = d\ast \omegaonedot = 
-d\stardot\omegaone = 
- d\ast S\omegaone.
$$
Hence $u \equiv -\hatPhi d* S\omegaone$ modulo the constant functions.
\end{proof}

In this paper we write 
\begin{eqnarray*}
&&\Omega_0 := \omegaone\wedge\omegaone \in A^2(C)\otimes H^{\otimes 2}
\\
&&K_0 := \hatPhi(\Omega_0) = \hatPhi(\omegaone\wedge\omegaone) \in
A^0(C)\otimes H^{\otimes 2}
\\
&&\nu_0 := \ast\dd K_0 = \ast\dd\hatPhi(\Omega_0) \in A^{1,
0}(C)\otimes H^{\otimes 2}.
\end{eqnarray*}
The first variation of $a_g$ is given by the following.
\begin{theorem}\label{203.3}
Let $\Xi$ be the quadratic differential defined by
\begin{eqnarray*}
\Xi &:=& M\left((\ast\dd K_0)
(\ast\dd K_0)
+4(\ast\dd\hatPhi(\ast d K_0\wedge\omegaone))\omega'_{(1)}\right)
\end{eqnarray*}
Then we have
\begin{equation*}
\mathop{a_g}^\centerdot = 
-\Re\left(4\sqrt{-1}\int_C\Xi\,\mu\right).
\end{equation*}
\end{theorem}

We denote by 
$$
\T := \prod^\infty_{m=0}H^{\otimes m}
$$
the completed tensor algebra generated by $H$. 
Let $\Omega = \{\Omega^t\}$, $t \in \bR$, 
$\vert t\vert\ll1$, be a family of $2$-forms with values in 
the algebra $\T$. Assume $A := \int_C\Omega \in \T$ is constant, 
and denote $K := \hatPhi(\Omega)$. Then we have

\begin{lemma}\label{204.1}
$$
(\int_CK\Omega)^\centerdot = \int_C(\mathop{\Omega}^\centerdot K +
K\mathop{\Omega}^\centerdot) - \int_C(\mathop{B}^\centerdot AK +
K\mathop{B}^\centerdot A) - \int_C(d\stardot dK)K.
$$
\end{lemma}
\begin{proof}
Since $\Omega = d\ast dK + BA$ and $\int_CKB = 0$, 
we have 
\begin{eqnarray*}
&&\int_C \mathop{K}^\centerdot\Omega 
= \int_C \mathop{K}^\centerdot d\ast dK 
+ \int_C \mathop{K}^\centerdot BA
= \int_C (d\ast d\mathop{K}^\centerdot) K 
- \int_C K \mathop{B}^\centerdot A\\
&=& \int_C(\mathop{\Omega}^\centerdot - \mathop{B}^\centerdot A 
- d\stardot dK)K - \int_CK\mathop{B}^\centerdot A.
\end{eqnarray*}
Hence 
\begin{eqnarray*}
&&(\int_CK\Omega)^\centerdot
= \int_C\mathop{K}^\centerdot\Omega 
+ \int_C K\mathop{\Omega}^\centerdot\\
&=& \int_C(\mathop{\Omega}^\centerdot K +
K\mathop{\Omega}^\centerdot) - \int_C(\mathop{B}^\centerdot AK +
K\mathop{B}^\centerdot A) - \int_C(d\stardot dK)K,
\end{eqnarray*}
as was to be shown.
\end{proof}

Let $\theta$ and $\varphi$ be real $1$-forms on $C$ with values in $\T$. 
Then we have
\begin{equation}
\int_C\ast \theta\wedge S\varphi 
= \Re\left(4\sqrt{-1}\int_C\theta'\varphi'\mu\right).
\label{205.1}
\end{equation}
In fact, we have 
\begin{eqnarray*}
&&\int_C\ast\theta\wedge S\varphi 
= -2\int_C\ast\theta\wedge(\varphi'\mu 
+ \varphi''\overline{\mu})
= 2\sqrt{-1}\int_C(\theta'-\theta'')\wedge
(\varphi'\mu + \varphi''\overline{\mu})\\
&=& 2\sqrt{-1}\int_C(\theta'\varphi')\mu
- 2\sqrt{-1}\int_C(\theta''\varphi'')\overline{\mu}
= \Re\left(4\sqrt{-1}\int_C(\theta'\varphi')\mu\right).
\end{eqnarray*}

\begin{proof}[Proof of Theorem \ref{203.3}]
Since $H'$ and $H''$ are isotropic in $H$, 
\begin{equation*}
M(\ast\dd\hatPhi(\nu_0\wedge\omega''_{(1)})\omega'_{(1)})
= M(\ast\dd\hatPhi(\ast d\hatPhi(\Omega_0)\wedge\omegaone)\omega'_{(1)}).
\end{equation*}
Hence we have 
\begin{equation}
\Xi = M\left(\nu_0\nu_0
+4(\ast\dd\hatPhi(\nu_0\wedge\omega''_{(1)}))\omega'_{(1)}\right)
\label{206.2}
\end{equation}
Now applying Lemma \ref{204.1} to $a_g = -M\int_CK_0\Omega_0$, 
we have 
\begin{equation}
-\mathop{a_g}^\centerdot = 2M\int_C K_0\mathop{\Omega_0}^\centerdot
- 2M\int_C\mathop{B}^\centerdot K_0I
-M\int_C(d\stardot dK_0)K_0.
\label{206.3}\end{equation}
The first term of the RHS in (\ref{206.3}) is 
\begin{eqnarray*}
&& 2M\int_C\hatPhi(\Omega_0)\omegaonedot\wedge\omegaone
+ 2M\int_C\hatPhi(\Omega_0)\omegaone\wedge\omegaonedot\\
&=& 4M\int_C\hatPhi(\Omega_0)\omegaone\wedge\omegaonedot
= 4M\int_C\hatPhi(\Omega_0)\omegaone\wedge d\hatPhi d\ast(-S\omegaone)\\
&=& -4M\int_C\ast
d\hatPhi(d\hatPhi(\omegaone\wedge\omegaone)\wedge\omegaone)\wedge
S\omegaone\\
&=& -4M\int_C\ast
d\hatPhi(d\hatPhi(\tast\omegaone\wedge\tast\omegaone)\wedge\tast\omegaone)\wedge
S\tast\omegaone\\
&=& 4M\int_C\ast\ast
d\hatPhi(\ast d\hatPhi(\Omega_0)\wedge\omegaone)\wedge
S\omegaone.
\end{eqnarray*}
From (\ref{201.5}) the second term vanishes. In fact,
$$
M\int_C\mathop{B}^\centerdot K_0I = m\int_C\mathop{B}^\centerdot K_0 
= 2g\int_C\mathop{B}^\centerdot\hatPhi(B) = 0.
$$
The third term is 
\begin{eqnarray*}
&&-M\int_C(d\stardot dK_0)K_0 = -M\int_C(\stardot dK_0)dK_0 = M\int_CdK_0(\stardot
dK_0)\\
&=& -M\int_C\ast dK_0\wedge \ast S\ast dK_0 = M\int_C\ast\ast dK_0 \wedge
S\ast dK_0.
\end{eqnarray*}
Hence from (\ref{205.1}) 
\begin{eqnarray*}
-\mathop{a_g}^\centerdot &=& M\left(\int_C\ast\ast dK_0\wedge S\ast dK_0
+  4\int_C\ast\ast d\hatPhi(\ast dK_0\wedge\omegaone)\wedge
S\omegaone\right)\\ &=& \Re\left(4\sqrt{-1}\int_C\Xi\,\mu\right),
\end{eqnarray*}
as was to be shown.
\end{proof}

Next we study the first variation of the Green function 
$h = -\hatPhi(\delta_{P_0})$.
Let $(C_t, P^t_0)$, $t\in \bR$, $\vert t\vert \ll 1$, a $C^\infty$ curve 
on the universal family $\Cg$. 
We can choose a family of diffeomorphisms $f^t: (C, P_0) \to (C_t,
P^t_0)$ such that $f^t$ is complex analytic near $P_0$ for sufficiently 
small $t$. Then $\mu = (\mu(f))^\centerdot$ vanishes near the point
$P_0$. \par
We compute $\hdot = \frac{d}{dt}\Bigr\vert_{t=0}{f^t}^*h_{P^t_0}$. 
Since ${f^t}^*\delta_{P^t_0} = \delta_{P_0}$ is constant, we have 
\begin{equation*}
d\ast d\hdot = \mathop{B}^\centerdot - d\stardot d h.
\end{equation*}
The RHS is smooth near $P_0$ 
since $\mu$ and $d\stardot dh$ vanish near $P_0$. 
Hence $\hdot$ is smooth near $P_0$ so that we may take the value 
of $\hdot$ at the point $P_0$, $\hdot(P_0)$. \par
From what we have discussed we may apply Lemma \ref{204.1}
to the $2$-current $\Omega = -\delta_{P_0}$. 
Then we have $\mathop{\delta_{P_0}}^{\centerdot} = 0$, $K = h_{P_0}$, 
$A = -1$ and $(\int_CK\Omega)^\centerdot =
\int_C\mathop{K}^\centerdot\Omega = -\hdot(P_0)$. 
Hence we obtain
\begin{equation}
\hdot(P_0) = -2\int_Ch\mathop{B}^\centerdot + \int_Ch(d\stardot dh).
\label{207.2}
\end{equation}

\begin{theorem}\label{207.3}
Let $\Upsilon$ be the quadratic differential defined by
\begin{equation*}
\Upsilon := (\ast\dd h)(\ast\dd h) + \frac{2}{g}\ast\dd\hatPhi(\ast
dh\wedge \omegaone)\cdot\omega'_{(1)}
\end{equation*}
Then we have
\begin{equation*}
\hdot(P_0) = 
-\Re\left(4\sqrt{-1}\int_C\Upsilon\,\mu\right).
\end{equation*}
\end{theorem}
\begin{proof}
We have ${\stackrel{\centerdot}{B}} =
\frac{1}{g}\omegaone\cdot\omegaonedot$ since 
$B = \frac{1}{2g}\omegaone\cdot\omegaone$. 
As to the first term in (\ref{207.2}) we have 
\begin{eqnarray*}
&& g\int_Ch\mathop{B}^\centerdot =
\int_Ch\omegaone\cdot\omegaonedot
= -\int_Ch\omegaone \cdot d\hatPhi d\ast S\omegaone\\
&=& -\int_C\ast d\hatPhi d(h\omegaone)\cdot S\omegaone
= \int_C\ast d\hatPhi(\ast dh\wedge\omegaone)\cdot S\tast\omegaone\\
&=& + \int_C\ast\ast d\hatPhi(\ast dh\wedge\omegaone)\cdot S\omegaone
= \Re\left(4\sqrt{-1}\int_C(\ast\dd\hatPhi(\ast
dh\wedge\omegaone)\cdot \omega'_{(1)})\mu\right).
\end{eqnarray*}
The last line follows from (\ref{205.1}).\par
The second term is equal to 
\begin{eqnarray*}
&&\int_Ch(d\stardot dh) = -\int_Cdh\wedge \stardot dh = \int_Cdh\wedge
S\ast dh\\
&=& -\int_C\ast\ast dh\wedge S\ast dh 
= -\Re\left(4\sqrt{-1}\int_C(\ast\dd h)(\ast\dd h)\mu\right).
\end{eqnarray*}
This completes the proof.
\end{proof}

\section{The flat connection for 
Johnson's homomorphisms.}\label{s-j}

Now we regard the Arakelov-Green function $G(P_0, P_1) := \exp(-4\pi
h_{P_0}(P_1))$ as a function on the fiber product $\Cg\times_\moduli\Cg$. 
The normal bundle of the diagonal map $\Cg \to \Cg\times_\moduli\Cg$ 
is equal to the relative tangent bundle $\TCM$. 
Hence the $(1, 1)$-form on $\Cg$
$$
e^A := \frac{1}{2\pi\sqrt{-1}}\dd\dc\log G\Bigr\vert_{\text{diagonal}}
$$
represents the first Chern class of the bundle $\TCM$. 
Since $\frac1{2\pi\sqrt{-1}}\log G(P_0, ) = 2\sqrt{-1}h_{P_0}$, 
we have
\begin{equation}
{e^A}_{[C, P_0]} = -2\sqrt{-1}\,\dc\dd h_{P_0}.\label{300.0}
\end{equation}

\par
In this section we review a flat connection introduced in the
previous paper \cite{Kaw2}, whose holonomy gives Johnson's
homomorphisms on the mapping class group. The connection induces 
a $(1, 1)$-form $e^J$ on $\Cg$ representing 
the first Chern class of $\TCM$ following Morita's recipe 
\cite{MoF}. 
We relate the second term of the connection form with 
the first variations ${\stackrel{\centerdot}{a_g}}$ and $\hdot(P_0)$ 
computed in \S\ref{s-fv}. As a consequence we obtain
\begin{theorem}\label{300.1}
$$
e^A - e^J = \frac{-2\sqrt{-1}}{2g(2g+1)}\dd\dc a_g.
$$
\end{theorem}

Let $[C, P_0] \in \Cg$ be a pointed Riemann surface. 
We define the Green operator $\Phi = \Phi^{(C, P_0)}: A^2(C) 
\to A^0(C)/\bC$ with respect to the delta current $\delta_{P_0}$
by 
\begin{equation}
d\ast d\Phi(\Omega) = \Omega - (\int_C\Omega)\,\delta_{P_0}.
\label{301.1}
\end{equation}
Here $A^0(C)/\bC$ is the quotient space by the constant functions
$\bC$. Since $d\Phi d= d\hatPhi d$, we have the Hodge decomposition 
\begin{equation}
\varphi = \harmonic\varphi + \ast d\Phi d\varphi + d\Phi d\ast
\varphi
\label{301.2}\end{equation}
for any $\varphi \in A^1(C)$.\par
We have 
\begin{equation}
\ast d\hatPhi(\Omega) = \ast d\Phi(\Omega) - (\int_C\Omega)\ast dh
\label{301.3}
\end{equation}
for any $\Omega \in A^2(C)$. In fact, 
$\Omega - (\int_C\Omega)\delta_{P_0}$ is $d$-exact, so that 
$$
\ast d\Phi(\Omega) - \ast d\hatPhi(\Omega)
= (\int_C\Omega)(\ast d\Phi - \ast d\hatPhi)(\delta_{P_0}) 
= (\int_C\Omega)\ast dh.
$$
In particular, substituting $\Omega = B$ into (\ref{301.3}), we obtain
\begin{equation}
\ast dh = \ast d\Phi(B) = \frac{1}{2g}m\ast d\Phi(\Omega_0)
\label{301.4}
\end{equation}

Now we define $\omega_{(n)}$, $n \geq 2$, by
\begin{equation*}
\omega_{(n)} := 
\ast d\Phi\left(\sum^{n-1}_{p=1}\omega_{(p)}
\wedge\omega_{(n-p)}\right)
\end{equation*}
inductively on $n$, and 
$\omega := \sum^\infty_{n=1}
\omega_{(n)} \in A^1(C)\wotimes\T$. 
Here $\wotimes$ means the completed tensor product. Then we have the (modified) integrability condition
\begin{equation*}
d\omega = \omega\wedge\omega - I\delta_{P_0},
\end{equation*}
where $I \in H^{\otimes 2}$ is the intersection form
in (\ref{201.3}). 
This means the $1$-form $\omega$ defines a flat connection on the vector bundle $(C\setminus \{P_0\})\times \T$. 
Its holonomy is the harmonic Magnus expansion, 
an embedding of the fundamental group of 
$C\setminus \{P_0\}$ with a tangential basepoint 
towards $P_0$ into the multiplicative group of the 
algebra $\T$ \cite{Kaw2}. \par

In the case $n=2$ we have 
\begin{equation}
\omega_{(2)} = \ast d\Phi(\Omega_0) = \ast d\hatPhi(\Omega_0) + I\ast dh
= \ast dK_0 + I\ast dh.
\label{303.1}
\end{equation}
from (\ref{301.3}).\par

\begin{lemma}\label{303.2}
$$
(m\otimes m + M)(\omega'_{(2)}\omega'_{(2)} +
2\omega'_{(3)}\omega'_{(1)}) = \Xi + 2g(2g+1)\Upsilon.
$$
\end{lemma}
\begin{proof} We have $mK_0 = 2g\hatPhi(B) = 0$ so that 
$m\omega_{(2)} = 2g\ast dh$. 
It follows from (\ref{303.1}) 
\begin{eqnarray*}
(m\otimes m + M)(\omega'_{(2)}\omega'_{(2)}) 
&=& 
M(\ast\dd K_0)(\ast\dd K_0) + (m\otimes m + M)(II)(\ast\dd h)(\ast\dd h)\\
&=& 
M(\ast\dd K_0)(\ast\dd K_0) + 2g(2g+1)(\ast\dd h)(\ast\dd h).
\end{eqnarray*}
Since $\omega_{(2)}\wedge\omega_{(1)}$ is $d$-exact,
$$
\omega'_{(3)} = \ast\dd\hatPhi(\omega_{(2)}\wedge\omega_{(1)}) 
+ \ast\dd\hatPhi(\omega_{(1)}\wedge\omega_{(2)}).
$$
Hence 
\begin{eqnarray*}
&& (m\otimes m + M)(\omega'_{(3)}\omega'_{(1)})\\
&=& 2(m\otimes m)(\ast\dd\hatPhi(\omega_{(2)}\wedge\omega_{(1)})
\omega'_{(1)}) + 2M(\ast\dd\hatPhi(\omega_{(2)}\wedge\omega_{(1)})
\omega'_{(1)})\\
&=& 2(2g+1)\ast\dd\hatPhi(\ast dh\wedge\omega_{(1)})\cdot
\omega'_{(1)}
+ 2M(\ast\dd\hatPhi(\ast dK_0\wedge\omega_{(1)})\omega'_{(1)}).
\end{eqnarray*}
This completes the proof.
\end{proof}

We introduce the operator $N: \T \to \T$ defined by 
\begin{eqnarray}
N\vert_{H^{\otimes n}} :=\sum^{n-1}_{k=0}\begin{pmatrix}
1& 2& \cdots & n-1 & n\\
2& 3& \cdots & n & 1
\end{pmatrix}^k.
\label{303.1}
\end{eqnarray}
As was shown in \cite[\S7]{Kaw2}, the covariant tensor 
 $N(\omega'\omega')$ is a meromorphic quadratic differential with a
unique pole at $P_0$, and so it is regarded as a $(1, 0)$-cotangent
vector at $[C, P_0, v]\in \Mgone$ in a natural way. 
We denote by $\eta'_n$ the $(n+2)$-nd graded component of 
$N(\omega'\omega')$. By abuse of notation we denote by 
$H^{\otimes n}$ the vector bundle 
$$
H^{\otimes n} := \coprod_{[C] \in \moduli}
H_1(C; \bC)^{\otimes n}
$$
over the moduli space $\moduli$. 
$\eta_n := \eta'_n +\overline{\eta'_n}$ is a twisted real 
$1$-form with values in the vector bundle $H^{\otimes (n+2)}$ 
on $\Mgone$. The real twisted $1$-form 
$\eta := \sum^\infty_{n=1}\eta_n$ induces a flat 
connection on $H^*\otimes \T_2$ and the holonomy gives 
all of Johnson's homomorphisms on the mapping class group
$\pi_1(\Mgone)$. Here we denote $\T_2 := \prod^\infty_{n=1}
H^{\otimes (n+2)}$.\par
Now we look at $\eta_2' = N(\omega'_{(1)}\omega'_{(3)} +
\omega'_{(2)}\omega'_{(2)} + \omega'_{(3)}\omega'_{(1)})$. 
By Lemma \ref{303.2} we have 
\begin{equation*}
(m\otimes m)\eta'_2 = \sqrt{-1}\dd a_g +2g(2g+1) \sqrt{-1}\dd h(P_0).
\end{equation*}
Hence $\dd(m\otimes m)\eta'_2 = 0$ and $d(m\otimes m)\eta_2 = 
2\dc(m\otimes m)\eta'_2$. Consequently
\begin{corollary}\label{303.5}
$$
d(m\otimes m)\eta_2 = -2\sqrt{-1}\dd\dc a_g - 2g(2g+1)e^A.
$$
\end{corollary}

From the flatness of $\eta$ we have $d\eta_1 = 0$. 
Moreover 
$\eta_1$ may be regarded as a real $1$-form on $\Cg$. 
We identify $\LaH$ with a submodule of $H^{\otimes 3}$ 
by the embedding
\begin{equation}
\LaH \to H^{\otimes 3}, \quad
Z_1\wedge Z_2\wedge Z_3 \mapsto
\sum_{\sigma\in\frak{S}_3}(\operatorname{sgn}\sigma)Z_{\sigma(1)}Z_{\sigma(2)}Z_{\sigma(3)}.
\label{304.1}
\end{equation}
As was proved in \cite[\S8]{Kaw2}, the cohomology class $-[\eta_1]$ is 
equal to the first extended Johnson homomorphism $\tilde k \in H^1(\Cg;
\LaH)$ introduced by Morita \cite{MoJ}. \par
We define $\hatM: H^{\otimes 6} \to \bC$ by 
\begin{equation}
\hatM(Z_1Z_2Z_3W_1W_2W_3) := (Z_1\cdot W_1)(Z_2\cdot W_2)(Z_3\cdot W_3), 
\quad Z_i, W_i \in H,
\label{304.2}
\end{equation}
$M_1$ and $M_2: \LaH\otimes\LaH \to \bC$ by
\begin{equation*}
M_1 := (m\otimes m\otimes m)\vert_{\LaH\otimes\LaH}
\quad \mbox{and}\quad
M_2 := \hatM\vert_{\LaH\otimes\LaH}
\end{equation*}
respectively. Then Morita \cite{MoF} proved
\begin{theorem}[\cite{MoF}, Theorems 5.1 and 5.8]
\begin{eqnarray*}
&&-\frac{1}{2g(2g+1)}(M_1+M_2)({\tilde k}^{\otimes 2}) = e (:= c_1(\TCM))
\in H^2(\Cg; \bC).\\
&&\frac{1}{2g+1}(-3M_1+2(g-1)M_2)({\tilde k}^{\otimes 2}) = e_1
\in H^2(\Cg; \bC).
\end{eqnarray*}
Here $e_1$ is the first Morita-Mumford class \cite{Mu} \cite{Mo}
on the moduli space $\moduli$. 
\end{theorem}

Following this theorem we define
\begin{eqnarray}
&&e^J := -\frac{1}{2g(2g+1)}(M_1+M_2)({\eta_1}^{\otimes 2})
\label{304.5}
\\
&&e^J_1 := \frac{1}{2g+1}(-3M_1+2(g-1)M_2)({\eta_1}^{\otimes 2}).
\label{304.6}
\end{eqnarray}
These are closed $2$-forms on $\Cg$ representing the cohomology classes 
$e$ and $e_1$, respectively. As will be shown in
\S\ref{s-e},
$e^J_1$ can be regarded as a $2$-form on $\moduli$. \par

Moreover we define $M_3$ and $M_4: H^{\otimes 6} \to H^{\otimes 4}$ by 
\begin{eqnarray}
&& M_3(Z_1Z_2Z_3W_1W_2W_3) := (Z_3\cdot W_1)Z_1Z_2W_2W_3,\\
&& M_4(Z_1Z_2Z_3W_1W_2W_3) := (Z_3\cdot W_2)W_1Z_1Z_2W_3,
\end{eqnarray}
respectively. Then we have
\begin{equation}
(m\otimes m)M_3\vert_{\LaH\otimes\LaH} = M_1, \quad\mbox{and}\quad
(m\otimes m)M_4\vert_{\LaH\otimes\LaH} = M_2.\label{304.7}
\end{equation}
It follows from the flatness of $\eta$ 
\begin{equation*}
d\eta_2 = (M_3+M_4)({\eta_1}^{\otimes 2})
\end{equation*}
\cite[Lemma 2.4]{Kaw2}. Hence we obtain
\begin{equation}
(m\otimes m)d\eta_2 = (M_1+M_2)({\eta_1}^{\otimes 2}) = 
-2g(2g+1)e^J.
\end{equation}
Theorem \ref{300.1} follows from Corollary \ref{303.5} and
(\ref{304.7}).\par

\bigskip
The residue of the quadratic differential $(m\otimes m)(\eta'_2)$ 
at the pole $P_0$ is $-\frac{1}{8\pi^2}2g(2g+1)$. This also implies 
$-\frac{1}{2g(2g+1)}(m\otimes m)d\eta_2$ represents the first 
Chern form of the relative tangent bundle $\TCM$.

\section{Integration along the fiber.}\label{s-int}

Now we introduce another $2$-form on $\moduli$ 
\begin{equation*}
e^F_1 := \fiber (e^J)^2
\end{equation*}
representing the first Morita-Mumford class $e_1 \in H^2(\moduli; \bC)$. 
To simplify the situation we compute
\begin{equation*}
E^F_1 := \frac{1}{4(2g-2)^4}\fiber M_1({\eta_1}^{\otimes 2})^2
= \frac14\fiber(m\otimes m)({\eta^H_1})^{\otimes 4}
\end{equation*}
instead of $e^F_1$. Here we denote
\begin{equation*}
\eta^H_1 := \frac{1}{2g-2}(m\otimes 1)\eta_1
\end{equation*}
It is clear that $\frac{1}{(2g-2)^2}M_1({\eta_1}^{\otimes 2}) =
m({\eta^H_1}^{\otimes 2})$. 
Since 
\begin{equation}
M_1({\eta_1}^{\otimes 2}) = -e^J_1 + 2g(2-2g)e^J,
\label{401.4}\end{equation}
we have
\begin{eqnarray*}
&& E^F_1 = \frac{g^2}{(2g-2)^2}e^F_1 - \frac{g}{(2g-2)^2}e^J_1
\\
&& [E^F_1] = \frac{g}{4(g-1)}e_1.
\end{eqnarray*}

Let $\lambda$ and $\mu$ be Beltrami differentials, or equivalently
elements in $C^\infty(C; TC\otimes\overline{T^*C})$.
$\lambda$ and $\mubar$ are regarded as tangent vectors of $\moduli$
at $[C]$. We define
\begin{eqnarray}
&& \elll := 2\hatPhi d\ast(\omega'_{(1)}\lambda) \in A^0(C)\otimes H'
\nonumber
\\
&& L^\lambda := \int_C\elll\omegaone\wedge\omegaone \in H^{\otimes
3}
\nonumber
\\
&& c^\lambda := \frac1{1-g}(m\otimes 1)(L^\lambda) = \frac1{1-g}
\int_C(\elll\cdot\omegaone)\omegaone \in H.\label{402.3}
\end{eqnarray}
The purpose of this section is to prove
\begin{theorem}\label{402.4}
The $2$-form $E^F_1$ is a $(1, 1)$-form on the moduli space $\moduli$, 
and we have
$$
E^F_1(\lambda, \mubar) = 2\int_C(\elll\cdot\omegaone)(\omegaone\cdot\ellm)
+ 2g\int_C(\elll\cdot\ellm)B + (2-2g)c^\lambda\cdot\overline{c^\mu}
$$
for any $\lambda$ and $\mu \in
C^\infty(C; TC\otimes\overline{T^*C})$. 
\end{theorem}

We denote by $q$ the $(1, 0)$-part of $-\eta^H_1$
\begin{equation*}
q := \frac{1}{2-2g}(m\otimes
1)N(\omega'_{(1)}\omega'_{(2)} + \omega'_{(2)}\omega'_{(1)}),
\end{equation*}
which is a meromorphic quadratic differential on $C$ with a unique 
pole at $P_0$.

\begin{lemma}\label{403.2}
$$
\int_Cq\lambda = \elll(P_0) + c^\lambda
$$
for any $\lambda \in C^\infty(C; TC\otimes\overline{T^*C})$. 
\end{lemma}
\begin{proof}
From (\ref{301.4}) and (\ref{303.1}) we have 
\begin{equation}
q = \frac{2}{1-g}(m\otimes 1)(\omega'_{(1)}\nu_0) - 2\omega'_{(1)}\ast\dd h.
\label{403.3}\end{equation}
In general, for any $2$-current $\Omega$ and smooth $1$-form $\varphi$,
we have
\begin{eqnarray}
\int_C\ast d\hatPhi(\Omega)\wedge\varphi &=& \int_C\Omega\hatPhi
d\ast\varphi\label{403.4}\\
\int_C\varphi\wedge\ast d\hatPhi(\Omega) &=& -\int_C(\hatPhi
d\ast\varphi)\Omega.\label{403.5}
\end{eqnarray}
Hence 
\begin{eqnarray*}
&&2(m\otimes 1)\int_C(\omega'_{(1)}\nu_0)\lambda = 
-2(m\otimes 1)\int_C(\omega'_{(1)}\lambda)\ast
d\hatPhi(\omegaone\wedge\omegaone)\\
&=& (m\otimes 1)\int_C\elll\omegaone\wedge\omegaone = (1-g)c^\lambda,
\end{eqnarray*}
$$
2\int_C(\omega'_{(1)}\ast\dd h)\lambda = 2\int_C(\omega'_{(1)}\lambda)\ast
d\hatPhi(\delta_{P_0}) = -\elll(P_0).
$$
Consequently $\int_Cq\lambda = c^\lambda + \elll(P_0)$, as was to be
shown.
\end{proof}

\begin{lemma}\label{403.6}
$$
\int_Cq\dc V' = (\omega'_{(1)}V')(P_0)
$$
for any $V' \in C^\infty(C; TC)$. In other words,
as a $(1, 0)$-form on $\Cg$, $q$ restricts to $\omega'_{(1)}$ on the 
fiber $C$ of the universal family $\pi: \Cg \to \moduli$.
\end{lemma}
\begin{proof}
Let $z$ be a complex coordinate on $C$ centered at $P_0$. 
We have $-2\ast\dd h \sim \frac{1}{2\pi\sqrt{-1}}\frac{dz}{z}$ 
near $P_0$. Since $q$ is integrable on $C$ and holomorphic on
$C\setminus\{P_0\}$,
\begin{eqnarray*}
\int_Cq\dc V' &=& -\lim_{\epsilon\downarrow0}\int_{\vert z\vert\geq
\epsilon} d(qV') = \lim_{\epsilon\downarrow0}\oint_{\vert z\vert=
\epsilon} qV'\\
&=& \frac1{2\pi\sqrt{-1}}\lim_{\epsilon\downarrow0}\oint_{\vert z\vert=
\epsilon} (\omega'_{(1)}V')\frac{dz}{z} = (\omega'_{(1)}V')(P_0).
\end{eqnarray*}
This proves the lemma.
\end{proof}
By (\ref{401.4}) we obtain
\begin{eqnarray}
e^J\vert_C = \frac{2-2g}{2g}m({\eta^H_1}^{\otimes 2})\vert_C = 
\frac{2-2g}{2g}m(\omegaone\wedge\omegaone)
= (2-2g)B.\label{403.65}
\end{eqnarray}
\begin{proof}[Proof of Theorem \ref{402.4}]
The $1$-forms $q$ and $\omega'_{(1)}$ have values in $H'$. 
Since $H'$ is isotropic, we have $m(\omega'_{(1)}\omega'_{(1)}) = 
m(\omega'_{(1)}q) = m(qq) = 0$. 
This means the $(2,0)$- and the $(0,2)$-parts of $E^F_1$ vanish.\par
From what we have discussed above it follows
\begin{eqnarray*}
&& E^F_1(\lambda, \mubar)\\
&=& (m\otimes m)\left(\int_C\left(\int_Cq\lambda\right)
\left(\int_C\overline{q}\mubar\right)\omegaone\omegaone
-2\int_C\left(\int_Cq\lambda\right)\omegaone
\left(\int_C\overline{q}\mubar\right)\omegaone
\right)\\
&=& \int_C\left(\int_Cq\lambda\right)\cdot(2\Omega_0 - 2gBI)\cdot
\left(\int_C\overline{q}\mubar\right)\\
&=& \int_C(\elll + c^\lambda)\cdot(2\Omega_0 - 2gBI)\cdot
(\ellm + \overline{c^\mu}).
\end{eqnarray*}
Since $\int_C\elll B = 0$, we have
\begin{eqnarray*}
&& c^\lambda\cdot(\int_C2\Omega_0 - 2gBI)\cdot \overline{c^\mu} =
c^\lambda\cdot (2-2g)I\cdot \overline{c^\mu} = (2g-2)c^\lambda\cdot
\overline{c^\mu},\\ 
&& c^\lambda\cdot(\int_C(2\Omega_0 - 2gBI)\cdot
\ellm) = 2c^\lambda\cdot \int_C\omegaone\wedge\omegaone\cdot\ellm
 = (2-2g)c^\lambda\cdot \overline{c^\mu},\\
&& (\int_C\elll\cdot(2\Omega_0 - 2gBI))\cdot \overline{c^\mu}
 = (2-2g)c^\lambda\cdot \overline{c^\mu}.
\end{eqnarray*}
Hence we obtain
$$
E^F_1(\lambda, \mubar) = 2\int_C(\elll\cdot\omegaone)(\omegaone\cdot\ellm)
+ 2g\int_C(\elll\cdot\ellm)B + (2-2g)c^\lambda\cdot\overline{c^\mu}.
$$
This completes the proof.
\end{proof}

In order to compare $E^F_1$ with the $2$-form $e^J_1$ we prove
\begin{lemma}\label{403.7}
$$
E^F_1(\lambda, \mubar) = -2\int_C\ellm\cdot(\elll\omega'_{(1)} - 
\omega'_{(1)}\elll)\cdot\omega''_{(1)} +
(2-2g)c^\lambda\cdot\overline{c^\mu}
$$
\end{lemma}
\begin{proof}
\begin{eqnarray*}
&& \int_C(\elll\cdot\omega_{(1)})(\omega_{(1)}\cdot\ellm) +
g\int_C(\elll\cdot\ellm)B\\
&=& \int_C(\ellm\cdot\omega'_{(1)})(\elll\cdot\omega''_{(1)}) -
\int_C(\ellm\cdot\elll)(\omega'_{(1)}\cdot\omega''_{(1)}) \\
&=& - \int_C\ellm\cdot(\elll\omega'_{(1)} - \omega'_{(1)}\elll)\cdot
\omega''_{(1)}.
\end{eqnarray*}
\end{proof}

\section{The $2$-form $e^J_1$.}\label{s-e}

In this section we compute $e^J_1(\lambda, \mubar)$ for 
$\lambda, \mu \in C^\infty(C; TC\otimes \overline{T^*C})$. 
We begin by a review on the module $\LaH$ and $Sp(H)$-invariant 
linear forms on $\LaH\otimes \LaH$. We regard $H$ as a submodule
of $\LaH$ through the injection $\qH: Z\in H \mapsto Z\wedge I = N(ZI)
\in \LaH$. If we define $\pH := \frac{1}{2g-2}(m\otimes 1)\vert_{\LaH}: 
\LaH \to H$, we have $\pH\qH = 1_H$. Following \cite{MoF} we write
\begin{equation*}
U := \operatorname{Coker}\qH = \LaH/H.
\end{equation*}
We denote the natural projection by $\pU: \LaH \to U$. 
The module $U$ is identified with $\operatorname{Ker}\pH \subset 
\LaH$. We denote by $\qU: U \to \operatorname{Ker}\pU \subset \LaH$ 
the natural injection. As was proved in \cite[\S8]{Kaw2}, 
$$
\eta^U_1 := \qU\pU\eta_1
$$
can be regarded as a $1$-form on the moduli space $\moduli$ with 
values in the vector bundle $\LaH$.\par
The map $\hatM: H^{\otimes 6}\to \bC$ in (\ref{304.2}) satisfies
\begin{equation}
\hatM((Z_1Z_2Z_3)(Z_4I)) = (Z_1\cdot Z_4)(Z_2\cdot Z_3)
\label{501.2}
\end{equation}
for any $Z_i \in H$. 
\begin{lemma}\label{501.3}
For any $z = Z_1\wedge Z_2\wedge Z_3$ and $w = W_1\wedge W_2\wedge W_3
\in \LaH$ we have 
$$
\hatM(\qU\pU z)(\qU\pU w) = (M_2 - \frac{3}{2g-2}M_1)(zw).
$$
\end{lemma}
\begin{proof} Denote $Z^H := \pH z$, $W^H := \pH w$, $z^H := \qH Z^H$ 
and $w^H := \qH W^H$. We have $\qU\pU z= z - z^H$ and $\qU\pU w= w -
w^H$. It is clear that $\hatM(zw) = M_2(zw)$. By straightforward 
computation using (\ref{501.2}) we obtain
\begin{eqnarray*}
&& \hatM(zw^H) = \hatM(z^Hw) = \hatM(N(Z^HI)w) = \frac{3}{2g-2}M_1(zw)
\quad \mbox{and}\\
&& \hatM(z^Hw^H) = 3\hatM(N(Z^HI)W^HI) = \frac{3}{2g-2}M_1(zw).
\end{eqnarray*}
Hence
$$
\hatM(\qU\pU z)(\qU\pU w) = \hatM(z-z^H)(w-w^H)
= M_2(zw) - \frac{3}{2g-2}M_1(zw),
$$
as was to be shown.
\end{proof}
By Lemma \ref{501.3} and (\ref{304.6}) we obtain
\begin{equation*}
e^J_1 = \frac{2g-2}{2g+1}\hatM((\eta^U_1)^{\otimes 2}).
\end{equation*}
Denote by $Q_0$ the $(1, 0)$-part of $\eta^U_1$
\begin{equation*}
Q_0 := N(\omega'_{(1)}\omega'_{(2)} + \omega'_{(2)}\omega'_{(1)}) 
+ N(qI),
\end{equation*}
which has values in ${\bigwedge}^2H'\wedge H''$. 
Since $H'$ and $H''$ are isotropic, $e^J_1$ is a $(1, 1)$-form.
We have 
\begin{equation}
\frac{2g+1}{2(2g-2)}e^J_1(\lambda, \mubar) 
= \hatM\left(\left(\int_CQ_0\lambda\right)
\left(\int_C\overline{Q_0}\mubar\right)\right)
\label{501.6}\end{equation}
for any $\lambda$ and $\mu \in C^\infty(C; TC\otimes\overline{T^*C})$. 

\begin{lemma}\label{502.1}
$$
\int_CQ_0\lambda = N(L^\lambda + c^\lambda I) 
\in \LaH \subset H^{\otimes 3}.
$$
\end{lemma}
\begin{proof} By (\ref{403.3}) and (\ref{303.1})
\begin{eqnarray*}
Q_0 &=& 2N(\omega'_{(1)}\omega'_{(2)} - \omega'_{(1)}(\ast\dd h)I
+ \frac1{1-g}\omega'_{(1)}\cdot \nu_0 I)\\
&=& 2N(\omega'_{(1)}\nu_0 + \frac1{1-g}\omega'_{(1)}\cdot \nu_0 I).
\end{eqnarray*}
Moreover
$$
2\int_C\omega'_{(1)}\nu_0\lambda = -2\int_C\omega'_{(1)}\lambda\ast
d\hatPhi(\Omega_0) = \int_C\elll\omegaone\wedge\omegaone = L^\lambda.
$$
Hence 
$$
\int_CQ_0\lambda = N(L^\lambda + \frac1{1-g}(m\otimes 1)(L^\lambda)I) 
= N(L^\lambda + c^\lambda I),
$$
as was to be shown.
\end{proof}

For simplicity we write 
\begin{equation*}
E^J_1 := \frac{2g+1}{6(2g-2)}e^J_1.
\end{equation*}
Then we have

\begin{lemma}\label{503.2}
$$
E^J_1(\lambda, \mubar) = \hatM((NL^\lambda)\overline{L^\mu}) + 
(2-2g)c^\lambda\cdot\overline{c^\mu}.
$$
\end{lemma}
\begin{proof} By (\ref{501.6}) we have
$$
E^J_1(\lambda, \mubar) = \hatM((NL^\lambda)\overline{L^\mu}) 
+ \hatM((NL^\lambda)\overline{c^\mu}I)
- \hatM((N\overline{L^\mu})c^\mu I) 
+ \hatM((Nc^\lambda I)\overline{c^\mu}I).
$$
The fourth term in the RHS is $(2g-2)c^\lambda\cdot\overline{c^\mu}$. 
From (\ref{501.2}) the second term is 
\begin{eqnarray*}
&&\hatM(N(\int_C\elll\omegaone\wedge\omegaone)\overline{c^\mu}I)\\
&=& (\int_C\elll(\omegaone\cdot\omegaone))\cdot\overline{c^\mu}
+ 2 \int_C(\elll\cdot\omegaone)(\omegaone\cdot\overline{c^\mu})\\
&=& 2g(\int_C\elll B)\cdot\overline{c^\mu} + 2((m\otimes
1)L^\lambda)\cdot\overline{c^\mu} = (2-2g)c^\lambda\cdot\overline{c^\mu}.
\end{eqnarray*}
Similarly the third term is equal to
$(2-2g)c^\lambda\cdot\overline{c^\mu}$. 
This proves the lemma.
\end{proof}

The amounts $\llll$ and $\clll$ depend only on $\lambda$ and the surface
$C$. This means $\eta^U_1$ and $e^J_1$ can be regarded as differential 
forms on the space $\Mg$. 

Moreover we obtain
\begin{proposition}\label{503.3}
$$
E^J_1(\lambda, \mubar) = 
-2\int_C\ellm\cdot\harmonic(\elll\omega'_{(1)}- \omega'_{(1)}\elll)\cdot
\omega''_{(1)}
 + (2-2g)c^\lambda\cdot\overline{c^\mu}.
$$
\end{proposition}
\begin{proof}
We have 
\begin{eqnarray*}
&& \hatM(L^\lambda\overline{L^\mu})\\
&=& \hatM(\int_C\elll\omega'_{(1)}\wedge\omega''_{(1)})
(\int_C\ellm\omega''_{(1)}\wedge\omega'_{(1)})
+ \hatM(\int_C\elll\omega''_{(1)}\wedge\omega'_{(1)})
(\int_C\ellm\omega'_{(1)}\wedge\omega''_{(1)})\\
&=& 2\hatM(\int_C\elll\omega'_{(1)}\wedge\omega''_{(1)})
(\int_C\ellm\omega''_{(1)}\wedge\omega'_{(1)})\\
&=& 2M(\int_C\elll\omega'_{(1)}\wedge\omega_{(1)})
\cdot(\int_C\omega_{(1)}\wedge\omega''_{(1)}\ellm)\\
&=& -2M\int_C\harmonic(\elll\omega'_{(1)})\omega''_{(1)}\ellm.
\end{eqnarray*}
Hence 
\begin{eqnarray*}
&& \hatM((NL^\lambda)\overline{L^\mu})\\
&=& \hatM(L^\lambda\overline{L^\mu}) -
2\hatM(\int_C\omega_{(1)}\elll\wedge\omegaone)
(\int_C\ellm\omegaone\wedge\omega_{(1)})\\
&=& \hatM(L^\lambda\overline{L^\mu}) -
2M(\int_C\omega'_{(1)}\elll\wedge\omegaone)
\cdot(\int_C\omegaone\wedge\omega''_{(1)}\ellm)\\
&=& \hatM(L^\lambda\overline{L^\mu}) + 
2M\int_C\harmonic(\omega'_{(1)}\elll)\omega''_{(1)}\ellm\\
&=& 
2M\int_C\harmonic(\omega'_{(1)}\elll -
\elll\omega'_{(1)})\omega''_{(1)}\ellm,
\end{eqnarray*}
as was to be shown.
\end{proof}

Finally we compute the $2$-form on $\moduli$
\begin{equation}
E^D_1 := E^F_1 - E^J_1 = \frac{g^2}{(2g-2)^2}e^F_1 - 
\frac{2g^2 + 2g -1}{3(2g-2)^2}e^J_1
\label{504.1}\end{equation}
representing $\frac1{12}e_1$. 
\begin{lemma}\label{504.2}
$$
E^D_1(\lambda, \mubar) = 4M\int_C\harmonic(\omega'_{(1)}\lambda)
\wedge\omega'_{(1)}\hatPhi d\ast(\omega''_{(1)}\ellm -
\ellm\omega''_{(1)}).
$$
\end{lemma}
\begin{proof}
By (\ref{106.3}) we have 
\begin{equation}\label{504.3}
\dd\elll = 2\dd\hatPhi d\ast(\omega'_{(1)}\lambda) = 
\omega'_{(1)}\lambda - \harmonic(\omega'_{(1)}\lambda).
\end{equation}
Hence by (\ref{106.5})
\begin{eqnarray*}
&&(1-\harmonic)(\elll\omega'_{(1)}-\omega'_{(1)}\elll)\\
&=& -2\sqrt{-1}\dd\hatPhi\dc(\elll\omega'_{(1)}-\omega'_{(1)}\elll)\\
&=& -2\sqrt{-1}\dd\hatPhi((\omega'_{(1)}\lambda -
\harmonic(\omega'_{(1)}\lambda))\omega'_{(1)}
+\omega'_{(1)}(\omega'_{(1)}\lambda
- \harmonic(\omega'_{(1)}\lambda)))\\
&=& 2\sqrt{-1}\dd\hatPhi(
\harmonic(\omega'_{(1)}\lambda)\omega'_{(1)}
+\omega'_{(1)}\harmonic(\omega'_{(1)}\lambda))
\end{eqnarray*}
From Lemmas \ref{403.7} and \ref{503.3} we have 
\begin{eqnarray*}
&& E^D_1(\lambda, \mubar) =
-2\int_C\ellm\cdot(1-\harmonic)(\elll\omega'_{(1)}-\omega'_{(1)}\elll)
\cdot\omega''_{(1)}\\
&=& 4\int_C\ellm\cdot\ast\dd\hatPhi(\harmonic(\omega'_{(1)}\lambda)\omega'_{(1)}
+\omega'_{(1)}\harmonic(\omega'_{(1)}\lambda))\cdot\omega''_{(1)}\\
&=&
4M\int_C\ast\dd\hatPhi(\harmonic(\omega'_{(1)}\lambda)\omega'_{(1)}
+\omega'_{(1)}\harmonic(\omega'_{(1)}\lambda))\omega''_{(1)}\ellm\\
&=&
4M\int_C\ast\dd\hatPhi(\harmonic(\omega'_{(1)}\lambda)\omega'_{(1)})
(\omega''_{(1)}\ellm-\ellm\omega''_{(1)})\\
&=&
4M\int_C\harmonic(\omega'_{(1)}\lambda)\omega'_{(1)}
\hatPhi d\ast(\omega''_{(1)}\ellm-\ellm\omega''_{(1)}).
\end{eqnarray*}
The last line follows from (\ref{403.4}).
\end{proof}

By similar computation we have 
\begin{equation}\label{505.1}
E^D_1(\lambda, \mubar) = 4\sqrt{-1}M
\int_C\harmonic(\omega'_{(1)}\lambda)\omega'_{(1)}
\hatPhi(\harmonic(\omega''_{(1)}\mubar)\omega''_{(1)}
+ \omega''_{(1)}\harmonic(\omega''_{(1)}\mubar)).
\end{equation}

\section{The second variation of the function $a_g$.}\label{s-sec}

This section is devoted to proving 
\begin{theorem}\label{603.2}
$$
\frac{-2\sqrt{-1}}{2g(2g+1)}\dd\dc a_g = \frac{1}{(2g-2)^2}(e^F_1 -
e^J_1).
$$
\end{theorem}

In the setting of (\ref{202.1}) we denote by $\stackrel{\circ}{\bigcirc}$ 
the antiholomorphic part of the variation
$\stackrel{\centerdot}{\bigcirc}$. By Theorem \ref{203.3} we have 
\begin{equation}\label{601.1}
(\dd\dc a_g)(\lambda,\mubar) =
2\sqrt{-1}\int_C\stackrel{\circ}{\Xi}\lambda
\end{equation}
for any $\lambda$ and $\mu \in C^\infty(C; TC\otimes \overline{T^*C})$. 
From (\ref{206.2}) the quadratic differential $\Xi$ is given by
\begin{eqnarray}
\Xi &=& M(\nu_0\nu_0) +
4M\ast\dd\hatPhi(\nu_0\wedge\omega''_{(1)})\omega'_{(1)}
\nonumber
\\
&=& M(\nu_0\nu_0) + 4M(\nu_1\omega'_{(1)}).\nonumber
\end{eqnarray}
Here we write simply 
\begin{equation*}
\nu_1 = \ast\dd\hatPhi(\nu_0\wedge\omega''_{(1)}).
\end{equation*}

 \par
From Lemma \ref{202.3} it follows
\begin{equation}
\stackrel{\circ}{\omegaone} = d\hatPhi d\ast(2\omega''_{(1)}\mubar) =
d\ellm.
\label{601.3}\end{equation}
Hence we have \begin{eqnarray}
\stackrel{\circ}{\Omega_0} &=& \stackrel{\circ}{\omegaone}\wedge\omegaone
+ \omegaone\wedge\stackrel{\circ}{\omegaone} = d(\ellm\omegaone -
\omegaone\ellm)\label{601.4}\\
\stackrel{\circ}{B} &=&
\frac1{2g}m\stackrel{\circ}{\Omega_0}
= \frac1{g}d(\ellm\cdot\omegaone).\label{601.5}
\end{eqnarray}
Let $\Omega = \{\Omega^t\}$, $t \in \bR$, 
$\vert t\vert\ll1$, be a family of $2$-forms with values in 
the algebra $\T$. Assume $A := \int_C\Omega \in \T$ is constant, 
and denote $\nu := \ast\dd\hatPhi(\Omega)$. Then we have
\begin{lemma}\label{602.1}
$$
\stackrel{\circ}{\nu} = \frac{1}{2}(\int_C\Omega\ellm)\cdot\omegaone
+ \ast\dd\hatPhi\stackrel{\circ}{\Omega} - \frac{1}{g}A\ast\dd\hatPhi
d(\ellm\cdot\omegaone).
$$
\end{lemma}
\begin{proof}
Differentiating $(\ast+\sqrt{-1})\nu = 0$, we get $0 = \stardot\nu +
(\ast+\sqrt{-1})\nudot = -2\sqrt{-1}\nu\mu + 2\sqrt{-1}(\nudot)''$. 
Hence $(\nudot)'' = \nu\mu$, and so $\numaru = 
(\numaru)' = \harmonic\numaru +
2\ast\dd\hatPhi\dc\stackrel{\circ}{\nu}$. \par
Since $\int_C\nu\wedge\omegaone = 0$ and $d\nu = \dc\nu = \frac12\Omega -
\frac12AB$, we have 
\begin{eqnarray*}
&&\harmonic\numaru = -(\int_C\numaru\wedge\omegaone)\cdot\omegaone
= (\int_C\nu\wedge\stackrel{\circ}{\omegaone})\cdot\omegaone\\
&=& (\int_C\nu\wedge d\ellm)\cdot\omegaone =
(\int_C(d\nu)\ellm)\cdot\omegaone =
\frac12(\int_C\Omega\ellm)\cdot\omegaone.
\end{eqnarray*}
The last line follows from $\int_CB\ellm = 0$. 
Hence we obtain
\begin{eqnarray*}
\numaru &=& \harmonic\numaru + 2\ast\dd\hatPhi\dc\stackrel{\circ}{\nu}
= \harmonic\numaru + \ast\dd\hatPhi(\stackrel{\circ}{\Omega}
- A\stackrel{\circ}{B})\\
&=& \frac12(\int_C\Omega\ellm)\cdot\omegaone 
+ \ast\dd\hatPhi\stackrel{\circ}{\Omega}
- \frac1{g}A\ast\dd\hatPhi d(\ellm\cdot\omegaone),
\end{eqnarray*}
as was to be shown.
\end{proof}

Differentiating $(\ast - \sqrt{-1})\omegaone = 
-2\sqrt{-1}\omega'_{(1)}$, we obtain from 
$\stackrel{\circ}{\ast}\omegaone=  2\sqrt{-1}\omega''_{(1)}\mubar$ and
(\ref{601.3})
\begin{eqnarray*}
&&-2\sqrt{-1}(\omega'_{(1)})^\circ 
= \stackrel{\circ}{\ast}\omegaone +
(\ast-\sqrt{-1})\stackrel{\circ}{\omegaone}\\
&=& 2\sqrt{-1}\omega''_{(1)}\mubar 
- 4\sqrt{-1}\dd\hatPhi d\ast(\omega''_{(1)}\mubar) 
= 2\sqrt{-1}\harmonic(\omega''_{(1)}\mubar),
\end{eqnarray*}
so that 
\begin{equation}
(\omega'_{(1)})^\circ = -\harmonic(\omega''_{(1)}\mubar).
\label{602.2}\end{equation}
Hence
\begin{equation}
M\int_C\nu_1(\omega'_{(1)})^\circ\lambda = 
-M\int_C\ast\dd\hatPhi(\nu_0\wedge\omega''_{(1)})
\harmonic(\omega''_{(1)}\mubar)\lambda = 0, 
\label{602.3}\end{equation}
since $H'$ and $H''$ are isotropic. 
Applying Lemma \ref{602.1} to $\nu_1$ we have
\begin{eqnarray}
&&4M\int_C(\nu_1\omega'_{(1)})^\circ\lambda =
4M\int_C\stackrel{\circ}{\nu_1}\omega'_{(1)}\lambda\nonumber\\
&=& 2M\left(\int_C\nu_0\wedge\omegaone\ellm\right)\cdot
\left(\int_C\omega'_{(1)}\omega'_{(1)}\lambda\right)\nonumber\\
&&+ 4M\int_C\ast\dd\hatPhi((\nu_0\wedge\omegaone)^\circ)
\omega'_{(1)}\lambda\nonumber\\
&=& 2M\left(\int_C\nu_0\wedge\omegaone\ellm\right)\cdot
\left(\int_C\omega'_{(1)}\omega'_{(1)}\lambda\right)
+ 2M\int_C(\nu_0\wedge\omegaone)^\circ\elll\nonumber\\
&=& 2M\left(\int_C\nu_0\wedge\omegaone\ellm\right)\cdot
\left(\int_C\omega'_{(1)}\omega'_{(1)}\lambda\right)
+ 2M\int_C\nu_0(d\ellm)\elll \label{602.4}\\
&&+ 2M\int_C\stackrel{\circ}
{\nu_0}\omegaone\elll\nonumber
\end{eqnarray}

Now we complute the third term $2M\int_C\stackrel{\circ}
{\nu_0}\omegaone\elll$. From Lemma \ref{602.1} applied to $\nu_0$
it follows
\begin{equation}
\stackrel{\circ}{\nu_0} =
\frac12\left(\int_C\omegaone\wedge\omegaone\ellm\right)\cdot\omegaone
+ \ast\dd\hatPhi d(\ellm\omegaone- \omegaone\ellm) - \frac1g
I\ast\dd\hatPhi d(\ellm\cdot\omegaone).
\label{602.5}\end{equation}
We have 
\begin{eqnarray*}
&&M\left(\int_C\omegaone\wedge\omegaone\ellm\right)
\cdot\left(\int_C\omegaone\wedge\omegaone\elll\right)
=-M(\llll\lllm) \\
&=& - \frac12\hatM(N\llll)\lllm + \frac12\hatM\llll\lllm.
\end{eqnarray*}
Since $H'$ and $H''$ are isotropic,
\begin{eqnarray*}
&& 2M\int_C\ast\dd\hatPhi d(\ellm\omegaone -
\omegaone\ellm)\omegaone\elll
= 2M\int_C\ast\dd\hatPhi d(\ellm\omega'_{(1)})\omega''_{(1)}\elll\\
&=& M\int_C\ellm\omega'_{(1)}\omega''_{(1)}\elll 
+ M\left(\int_C\ellm\omega'_{(1)}
\omegaone\right)
\cdot\left(\int_C\omegaone\omega''_{(1)}\elll\right)\\
&=& g\int_C\elll\cdot\ellm B - \frac12\hatM(\llll\lllm).
\end{eqnarray*}

On the other hand, from (\ref{402.3}),
\begin{eqnarray*}
&&(g-1)^2\cllm\cdot\clll =
\left(\int_C(\ellm\cdot\omegaone)\wedge\omegaone\right)\cdot
\left(\int_C\omegaone\wedge(\omegaone\cdot\elll)\right)\\
&=& - \int_C\harmonic(\ellm\cdot\omegaone)\omegaone\cdot\elll\\
&=& - \int_C(\ellm\cdot\omegaone)(\omegaone\cdot\elll) 
+ 2\int_C\ast\dd\hatPhi d(\ellm\cdot\omegaone)\omegaone\cdot\elll
\end{eqnarray*}
Hence 
\begin{eqnarray*}
&& -\frac2g M \int_CI\ast\dd\hatPhi d(\ellm\cdot\omegaone)\omegaone\elll 
= -\frac2g\int_C\ast\dd\hatPhi d(\ellm\cdot\omegaone)\omegaone\elll\\
&=& \frac{(g-1)^2}{g}\clll\cdot\cllm +
\frac1g\int_C(\elll\cdot\omegaone)(\omegaone\cdot\ellm).
\end{eqnarray*}
Consequently we obtain
\begin{eqnarray}
&&2M\int_C\stackrel{\circ}{\nu_0}\omegaone\elll\label{602.6}\\
&=& -\frac12\hatM(N\llll)\lllm + g\int_C\elll\cdot\ellm B 
+ \frac{(g-1)^2}{g}\clll\cdot\cllm + 
\frac1g\int_C(\elll\cdot\omegaone)(\omegaone\cdot\ellm).\nonumber
\end{eqnarray}

Next we compute $\int_CM(\nu_0\nu_0)^\circ\lambda$.
Here we remark $M(I\nu_0) = m\nu_0 =
m\ast\dd\hatPhi(\omegaone\wedge\omegaone) = 
2g\ast\dd\hatPhi B = 0$. From (\ref{602.5}) 
\begin{eqnarray*}
&& M(\nu_0\nu_0)^\circ = 2M(\stackrel{\circ}{\nu_0}\nu_0)\\
&=& M(\left(\int_C\omegaone\wedge\omegaone\ellm\right)
\cdot\omegaone\nu_0)
+ 2M(\ast\dd\hatPhi d(\ellm\omegaone-\omegaone\ellm)\nu_0).
\end{eqnarray*}
The second term is equal to 
\begin{eqnarray*}
&& 4M(\ast\hatPhi d(\ellm\omega'_{(1)})\nu_0)
= 2M(\ellm\omega'_{(1)}\nu_0) - 2M(\harmonic(\ellm\omega'_{(1)})\nu_0)\\
&=& 2M(\nu_0\ellm\omega'_{(1)}) +
2M(\left(\int_C\ellm\omega'_{(1)}\omega''_{(1)})\right)
\cdot\omega'_{(1)}\nu_0).
\end{eqnarray*}
Since $\omega'_{(1)}\lambda = \harmonic(\omega'_{(1)}\lambda) +
\dc\elll$, we have 
\begin{eqnarray}
&&\int_CM(\nu_0\nu_0)^\circ\lambda\nonumber\\
&=& M\left(\int_C\omegaone\wedge\omegaone\ellm\right)
\cdot\int_C\omegaone\nu_0\lambda 
+ 2M\int_C\nu_0\ellm\harmonic(\omega'_{(1)}\lambda)\label{603.1}\\
&&+ 2M\int_C\nu_0\ellm d\elll 
+ 2M(\left(\int_C\ellm\omega'_{(1)}\omega''_{(1)})\right)
\cdot\left(\int_C\omega'_{(1)}\nu_0\lambda\right)).\nonumber
\end{eqnarray}
By (\ref{403.5})
$$
2\int_C\omegaone\nu_0\lambda = 2\int_C\omega'_{(1)}\nu_0\lambda
= -2\int_C\omega'_{(1)}\lambda\nu_0 
= \int_C\elll\omegaone\wedge\omegaone = \llll.
$$
Hence the sum of the first and the fourth terms in (\ref{603.1}) is
$$
\frac12\hatM(\lllm\llll) + M(\lllm\cdot\llll) 
= -\frac12\hatM(N\llll)\lllm.
$$
The second term in (\ref{602.4}) and the third in (\ref{603.1}) are
\begin{eqnarray*}
&&2M\int_C\nu_0(d\ellm)\elll + 2M\int_C\nu_0\ellm d\elll 
= 2M\int_C (d\nu_0)\ellm\elll\\
&=& M\int_C\omegaone\wedge\omegaone\ellm\elll - M\int_CIB\ellm\elll\\
&=& \int_C(\elll\cdot\omegaone)(\omegaone\cdot\ellm) 
+ \int_C\elll\cdot\ellm B.
\end{eqnarray*}
The first in (\ref{602.4}) and the second in (\ref{603.1}) are
\begin{eqnarray*}
&& 2M\left(\int_C\nu_0\omegaone\ellm\right)\cdot
\left(\int_C\omega'_{(1)}\omega'_{(1)}\lambda\right) + 
2M\int_C\nu_0\ellm\harmonic(\omega'_{(1)}\lambda)\\
&=& 2M\left(\int_C\nu_0(\omegaone\ellm - \ellm\omegaone)\right)
\cdot\left(\int_C\omega'_{(1)}\omega'_{(1)}\lambda\right)\\
&=& 2M\left(\int_C\omegaone\wedge\omegaone\hatPhi
d\ast(\omega''_{(1)}\ellm -
\ellm\omega''_{(1)})\right)
\cdot\left(\int_C\omega'_{(1)}\omega'_{(1)}\lambda\right)\\
&=& 2M\left(\int_C\omega'_{(1)}\omega'_{(1)}\lambda\right)\cdot
\int_C\omegaone\wedge\omegaone\hatPhi
d\ast(\omega''_{(1)}\ellm -\ellm\omega''_{(1)})\\
&=& 2M\int_C\harmonic(\omega'_{(1)}\lambda)\omega'_{(1)}
\hatPhi d\ast(\omega''_{(1)}\ellm -\ellm\omega''_{(1)})\\
&=& \frac12 E^D_1(\lambda, \mubar).
\end{eqnarray*}
Consequently we obtain
\begin{eqnarray*}
&& \int_C\stackrel{\circ}{\Xi}\lambda\\
&=& -\hatM(N\llll)\lllm + (g+1)\int_C\elll\cdot\ellm B
+\frac{g+1}{g}\int_C(\elll\cdot\omegaone)(\omegaone\cdot\ellm)\\
&& + \frac12E^D_1(\lambda, \mubar) + \frac{(g-1)^2}{g}\clll\cdot\cllm\\
&=& -E_1^J(\lambda, \mubar) + \frac{2g+1}{2g}E^F_1(\lambda, \mubar)
+ \frac12E^D_1(\lambda, \mubar)\\
&=& \frac{(2g+1)2g}{4(2g-2)^2}(e^F_1 - e^J_1),
\end{eqnarray*}
which means
$$
\frac{-2\sqrt{-1}}{2g(2g+1)}\dd\dc a_g = \frac{1}{(2g-2)^2}(e^F_1 -
e^J_1).
$$
This completes the proof of Theorem \ref{603.2}.


\bibliographystyle{amsplain}

\end{document}